\documentclass[reqno]{amsart}
\usepackage{amsfonts,amsmath,times}
\usepackage{mhequ}

\newtheorem{theorem}{Theorem}[section]
\newtheorem{lemma}[theorem]{Lemma}

\theoremstyle{definition}

\theoremstyle{remark}
\newtheorem{remark}[theorem]{Remark}

\numberwithin{equation}{section}

\begin{document}

\newcommand{\tmname}[1]{\textsc{#1}}
\newcommand{\tmop}[1]{\operatorname{#1}}
\newcommand{\tmsamp}[1]{\textsf{#1}}
\newenvironment{enumerateroman}{\begin{enumerate}[i.]}{\end{enumerate}}
\newenvironment{enumerateromancap}{\begin{enumerate}[I.]}{\end{enumerate}}

\newcounter{problemnr}
\setcounter{problemnr}{0}
\newenvironment{problem}{\medskip
  \refstepcounter{problemnr}\small{\bf\noindent Problem~\arabic{problemnr}\ }}{\normalsize}
\newenvironment{enumeratealphacap}{\begin{enumerate}[A.]}{\end{enumerate}}
\newcommand{\tmmathbf}[1]{\boldsymbol{#1}}

\def\paral{/\kern-0.55ex/}
\def\parals_#1{/\kern-0.55ex/_{\!#1}}
\def\bparals_#1{\breve{/\kern-0.55ex/_{\!#1}}}
\def\n#1{|\kern-0.24em|\kern-0.24em|#1|\kern-0.24em|\kern-0.24em|}

\newcommand{\A}{{\bf \mathcal A}}
\newcommand{\B}{{\bf \mathcal B}}
\def\C{\mathbb C}
\def\D{\mathcal D}
\newcommand{\dom}{{\mathcal D}om}
\newcommand{\pathR}{{\mathcal{\rm I\!R}}}
\newcommand{\Nabla}{{\bf \nabla}}
\newcommand{\E}{{\mathbf E}}
\newcommand{\Epsilon}{{\mathcal E}}
\newcommand{\F}{{\mathcal F}}
\newcommand{\G}{{\mathcal G}}
\def\g{{\mathfrak g}}
\newcommand{\HH}{{\mathcal H}}
\def\h{{\mathfrak h}}
\def\k{{\mathfrak k}}
\newcommand{\I}{{\mathcal I}}
\def\LL{{\mathbb L}}
\def\law{\mathop{\mathrm{ Law}}}
\def\m{{\mathfrak m}}
\newcommand{\K}{{\mathcal K}}
\newcommand{\p}{{\mathfrak p}}
\newcommand{\R}{{\mathbb R}}
\newcommand{\Rc}{{\mathcal R}}
\def\T{{\mathcal T}}
\def\M{{\mathcal M}}
\def\N{{\mathcal N}}
\newcommand{\pnabla}{{\nabla\!\!\!\!\!\!\nabla}}
\def\X{{\mathbb X}}
\def\Y{{\mathbb Y}}
\def\L{{\mathcal L}}
\def\1{{\mathbf 1}}
\def\half{{ \frac{1}{2} }}
\def\vol{{\mathop {\rm vol}}}
\def\euc{{\mathop {\rm eul}}}

\def\ad{{\mathop {\rm ad}}}
\def\Conj{{\mathop {\rm Ad}}}
\def\Ad{{\mathop {\rm Ad}}}
\newcommand{\const}{\rm {const.}}
\newcommand{\eg}{\textit{e.g. }}
\newcommand{\as}{\textit{a.s. }}
\newcommand{\ie}{\textit{i.e. }}
\def\s.t.{\mathop {\rm s.t.}}
\def\esssup{\mathop{\rm ess\; sup}}
\def\Ric{\mathop{\rm Ric}}
\def\div{\mathop{\rm div}}
\def\ker{\mathop{\rm ker}}
\def\Hess{\mathop{\rm Hess}}
\def\Image{\mathop{\rm Image}}
\def\Dom{\mathop{\rm Dom}}
\def\id{\mathop {\rm Id}}
\def\Image{\mathop{\rm Image}}
\def\Cyl{\mathop {\rm Cyl}}
\def\Conj{\mathop {\rm Conj}}
\def\Span{\mathop {\rm Span}}
\def\trace{\mathop{\rm trace}}
\def\ev{\mathop {\rm ev}}
\def\supp{{\mathrm supp}}
\def\Ent{\mathop {\rm Ent}}
\def\tr{\mathop {\rm tr}}
\def\graph{\mathop {\rm graph}}
\def\loc{\mathop{\rm loc}}
\def\so{{\mathfrak {so}}}
\def\su{{\mathfrak {su}}}
\def\u{{\mathfrak {u}}}
\def\o{{\mathfrak {o}}}
\def\pp{{\mathfrak p}}
\def\gl{{\mathfrak gl}}
\def\hol{{\mathfrak hol}}
\def\z{{\mathfrak z}}
\def\t{{\mathfrak t}}
\def\<{\langle}
\def\>{\rangle}
\def\span{{\mathop{\rm span}}}
\def\diam{\mathrm {diam}}
\def\inj{\mathrm {inj}}
\def\Lip{\mathrm {Lip}}
\def\Iso{\mathrm {Iso}}
\def\Osc{\mathop{\rm Osc}}
\renewcommand{\thefootnote}{}
\def\supp{\mathrm {Supp}}
\title{ On Hypoelliptic Bridge}

\author{Xue-Mei Li}
%\footnote{printed \today}
\maketitle

\begin{abstract}
We prove that if the Markov generator of a diffusion process satisfies the two step strong H\"ormander condition, the conditioned hypoelliptic bridge satisfies an integral bound and is a continuous semi-martingale.  \end{abstract}

\footnote{AMS Mathematics Subject Classification : 60Gxx, 60Hxx, 58J65, 58J70. 
Part of the work is done during a visit to MSRI, 2015. {Address:  Mathematics Institute, The University of Warwick, Coventry CV4 7AL, U.K.}}

\section{Introduction}
We are motivated by the path integration formula  and also by  the $L^2$ analysis on the space of pinned continuous curves where
the Brownian bridge  plays an important role. 
Let $M$ be a smooth connected Riemannian manifold. Denote $C([0,1];M)$ the space of 
continuous functions from $[0,1]$ to $M$ and $C_{x_0, z_0}([0,1];M)$ its subspace of curves that begin at $x_0$ and end at $z_0$.   
 If $(x_t)$ is a Brownian motion with initial value $x_0$ and without explosion, a Brownian bridge $(b_t^{x_0,z_0}, 0\le t \le 1)$ begins with $x_0$ and ends at $z_0$ is a stochastic process with probability distribution $P(\cdot| x_1=z_0)$. The Brownian bridge is well known to induce a probability measure on $C_{x_0, z_0}([0,1];M)$ for $M$ compact.

For the $L^2$ analysis, it is standard to equip the space with the probability measure determined by the Brownian bridge, which fuelled the study of the logarithm of the heat kernel and their derivatives.
However there is no particular strong argument for the use of Brownian bridges, and indeed
one is tempted to explore. For example on a Lie group, a basic object is
a diffusion operator built from a family of left invariant vector fields. 

If $\{X_i, i= 0, 1, \dots, m\}$ is a family of smooth vector fields,
let $\L={1\over 2}\sum_{k=1}^m L_{X_k}L_{X_k}+L_{X_0}$  where $L_v$ denotes Lie differentiation in the direction of a vector $v$. 
If the diffusion coefficients $\{X_1, \dots, X_m\}$ and their iterated Lie brackets span the tangent space $T_xM$ at each $x$, $\L$ is said to satisfy {\it the strong H\"ormander condition}. Denote $D_k$ the set of vector fields and their commutators up to level $k$. If $\L$ satisfies the strong H\"ormander condition the minimal $k$ needed to span $T_xM$ is denoted by $l(x)$. % and the length increase by $2$ whenever the drift vector field is used.
If for all $x$, $l(x)\le p$, $\L$ is said to satisfy the {\it $p$-step strong H\"ormander condition}. 
We assume that there exists a global parabolic integral kernel for $\L$, which holds if
  $\L$ is a sub-Laplacian and the sub-Riemannian distance is complete, L. Strichartz \cite{Strichartz-sub-riemannian-geom};
or is symmetric and $M$ is compact, D. Jerison and A. Sanchez-Calle \cite{Jerison-Sanchez-Calle} and B. Davies \cite{Davies-sum-squares-88}; or is uniformly hypoelliptic and $M=\R^n$, S. Kusuoka and D. Stroock \cite{Kusuoka-Stroock-I}. See aso  L. Rothschild and E. Stein
\cite{Rothschild-Stein} and G. B. Folland \cite{Folland-subelliptic}. 
 
Suppose that $\L$ satisfies the two step strong H\"ormander condition and $X_0=\sum_{k=1}^m c_k X_k$. We make the following simple observation on the hypoelliptic bridge $(y_t)$:
$$\E \int_0^1 \left|d \log q_{1-s} (\cdot, z_0)(X_k(y_s))\right| ds<\infty,$$
 in particular $(y_t, 0\le t\le 1)$ is a continuous semi-martingale. The integral bound is obtained from small time estimates on the fundamental solution and its gradient, the latter from H. Cao and S.-T. Yau \cite{Cao-Yau}.  The Gaussian bounds for $q_t$ depend on the volume of the intrinsic metric balls $B_x(\sqrt t)$ for small time and on the Euclidean ball for large time. Around $x$ the metric distance is comparable with $\rho^{1\over l(x)}$ where $\rho$ is the Riemannian distance. The larger is $l(x)$, the more singular is the heat kernel at $0$.
Hence the semi-martingale property for diffusions satisfying two-step H\"ormander condition does not hint for a generalisation.
It is tempting to argue that this property fails when $l(x)$ is sufficiently large. On the other hand the following results are proved recently: the Brownian bridge concentrates on the sub-Riemannian geodesic at $t\to 0$. See I. Bailleul, L. Mesnager and J. Norris \cite{Bailleul-Mesnager-Norris}, and Y. Inahama \cite{Inahama}. Since the semi-martingale property depends on properties of the heat kernel for small time,
and since the sub-Riemannian geodesic is horizontal in whose direction the singularity in $t$ should be exactly $t^{-{n\over 2}}$, we tend to believe this semi-martingale property holds much more generally. 

\section{Preliminaries}
The purpose for this section is to familiarize ourselves with the basic properties of hypoelliptic bridges.
To condition a diffusion process from $x_0$ to reach $y_0$ at $1$, it is natural to assume there is a control path reaching $y_0$ from $x_0$ and the transition probability measures have positive densities, $q_t$, with respect to a Riemannian volume measure.
Hence it is reasonable to assume H\"ormander condition. It is well known that, at least when $M$ is a compact manifold,
 the conditioned diffusion induces a measure on the space of continuous paths. This is noted in J. Eells and K. D. Elworthy 
\cite{Eells-Elworthy},   J.-M. Bismut \cite{Bismut-martingale},   P.  Malliavin and M.-P. Malliavin \cite{Malliavin-Malliavin},  and B. Driver \cite{Driver-bb}.
 Eells and Elworthy were interested in relating the Wiener and pinned Wiener measures to the topology and geometry of the 
path space over a manifold,
which later involving the quest for an $L^2$ Hodge theory, see e.g. K. D. Elworthy and Xue-Mei Li \cite{Elworthy-Li-l2, Elworthy-Li-icm}.
Bismut, Driver, Malliavin and Malliavin were interested in the quasi-invariance of the pinned Brownian motion measure. An alternative
proof for the quasi-invariance theorem of  Malliavin and Malliavin was given  in M. Gordina \cite{Gordina-quasi}.

We discuss two cases: in the first $\L$ has an invariant measure $\mu$, i.e. $\int \L f d\mu=0$ for any $f\in C_K^\infty$,
and in the second we assume estimates on the heat kernel. We begin with the first case.
In general we do not know  there is a global solution to $\L^* \mu=0$.  
If $\L$ satisfies strong H\"ormander condition, and $M$ is compact or $\L$ is symmetric, the $\L$-diffusion $(x_t)$ has
an invariant measure.
 If $\{X_1, \dots, X_m\}$ are linearly independent they determine a sub-Riemannian metric. 
The sub-elliptic Laplacian $\Delta_H$ is defined to be $ \trace \div \nabla^H$ where $\nabla^H$ is the sub
Riemannian gradient and the divergence is with respect to a volume form $\mu$. Then
$\Delta_H=\sum_{i=1}^m L_{X_i}L_{X_i}+X_0$ where $X_0=-\sum_{i=1}^m \div_\mu(X_i)X_i $. If the sub-Riemannian metric is complete,
 then $\Delta_H$ is essentially self adjoint on $L^2(M; \mu)$,  see  R. Strichartz \cite{Strichartz-sub-riemannian-geom}.   
  If the sub-Riemannian metric is the restriction of a complete Riemannian metric or if the sub-Riemannian structure is obtained from left invariant vector fields on a lie group,  $M$ is a complete metric space with respect to the sub Riemannian distance $d$. In this paper we do not
  use a sub-Riemannian structure and will however comment on this at the end of the paper.

Throughout this paper $x_t$ is assumed to be conservative, otherwise the set of paths considered would exclude the
paths with life time less than $1$, which we are not willing to compromise. For simplicity we drop the subscript $1$ in $q_1$.
 If $f:M\to \R$ is a differentiable function we define
its horizontal gradient to be
$\nabla^H f=\sum_{i=1}^m X_i df(X_i)$.
Let $\hat \L$ denote the adjoint operator with respect to a, not necessarily finite, invariant measure $\mu$, i.e. $\int \L f g d\mu=\int f\hat \L g d\mu$. 
Denote  $\hat x_t$ the adjoint process. 

\begin{lemma}
\label{theorem1}
If  $\L^* \mu=0$ has a solution and the adjoint process is conservative,
the hypoelliptic bridge determines a probability measure 
on $C_{x_0,y_0}([0,1];M)$.
\end{lemma}
\begin{proof}
Let $(x_t)$ be an $\L$-diffusion and $(y_t)$ the conditioned bridge process.  
Restricted to an interval $[0, 3/4]$,  $y_t$ is a `Doob transform' of $(x_t)$. Let $\{w_t^i\}$ be a family of real valued independent one dimensional Brownian motions. Then $x_t$ and $y_t$ can be represented as solutions to the equations with initial values $x_0=y_0$,
$$dx_t =\sum_{i=1}^m X_i(x_t) \circ dw_t^i+ X_0(x_t)dt,$$
$$dy_t =\sum_{i=1}^m X_i(y_t) \circ dw_t^i+ X_0(y_t)dt +\nabla^H \log q_{1-t} (y_t, y_0)dt,$$
where the gradient is with respect to the first variable.

It is easy to see that $\exp({N_t})={ q_{1-t}(x_t, y_0)\over q(x_0, y_0)}$.
Since $\E { q_{1-t}(x_t, y_0)\over q(x_0, y_0)}=1$, $\exp({N_s}, 0\le s\le t)$ is a martingale.
If $F$ is supported on continuous paths defined up to a time $t<1$, then
$\E F(y_{\cdot} )=\E F(x_{\cdot}) e^{N_t}$. 
From this and $\E N_t=1$, we see that the finite dimensional distributions of $(y_t)$ agree with that of the
conditioned process, when restricted to $[0,t]$.
  Since $(x_t)$ admits a continuous modification and hence determines a probability measure on $C([0,3/4];M)$, 
so does $(y_t)$. 

The invariant measure $\pi$ is a distributional solution to $\L^*m=0$ where 
$$\L^*=\sum_{i=1}^mL_{X_i}L_{X_i}- L_{X_0} 
   -2\sum_{i=1}^m\div(X_i) L_{X_i}
+\div(X_0)  -\sum_{i=1}^m d(\div (X_i))(X_i) $$
is the $L^2$ adjoint of $\L$ with respect to the volume measure, with respect to which the divergence is also taken. Then $\L^*$
satisfies also the strong H\"ormander condition. By the standard theory, see L. Hormander \cite{Hormander-hypo-acta},  $\mu$ has a strictly positive smooth density $m$.

If $\hat x_t$ is adjoint to $(x_t)$, with respect to $m$, its Markov generator has the same leading term as $\L$
and satisfies also strong H\"ormander condition. We denote by $\hat q_t$ its smooth density and there is the following identity:
$m(x)q_t(x,y)=m(y)\hat q_t(y,x)$.
Since the $\hat \L$ diffusion does not explode, we condition $\hat x_t$ to reach $x$ from $y$ in time $1$. The corresponding process 
is denoted by $\hat y_t$. Then
$\hat y_{1-t}$ has the same distribution as $y_t$. This follows from 
 $$q^{ x_0, y_0}_{t_1, \dots, t_n}={ q_{t_1}(x_0, x_1) \dots q_{t_n-t_{n-1}}(x_{n-1}, x_n)q_{1-t_n}(x_n, y_0)\over q(x_0, y_0)},  \quad t_i<1,$$
 in which we replace $q$ by $\hat q$.
By the same argument as above, we see that $\hat y_t$ has a continuous modification on $[0, 3/4]$. Thus $x_t$
determines a probability measure on $C_{x_0, y_0}([0,1];M)$.  The probability measure on the Borel $\sigma$-algebra of $M^{[0,1]}$,
 agrees with those determined by the continuous modification
of $x_t$, when restricted to paths on $[0, 3/4]$ and $[1/4, 1]$. The required conclusion follows.
\end{proof}

We move on to results based on heat kernel estimates and begin with reviewing Gaussian upper bounds for the fundamental solutions. The Markov generator for an elliptic diffusion is necessarily of the form ${1\over 2}\Delta+Z$ where $\Delta$ is the Laplace-Beltrami operator for some Riemannian metric on $M$ and $Z$ is a vector field, in which case the diffusion is a Brownian motion with drift $Z$. Once we understand the case of $\L={1\over 2}\Delta$, an additional (well behaved)  drift vector field $Z$ can be taken care of.  For a detailed review on heat kernel upper bounds  see L. Saloff-Coste \cite{Saloff-Coste-heat-kernel}.
Take first $\L={1\over 2}\Delta$. If the Ricci curvature of the manifold is bounded from below by $-K$ where $K$ is a positive number, then $p_t(x,x)\sim t^{-{n\over 2}}$ where $n=\dim(M)$ and $t\in (0,1)$.
This is a theorem of P. Li and S.-T. Yau \cite{Li-Yau},  extending the result of J. Cheeger and S.-T. Yau
\cite{Cheeger-Yau}.  In general if there exists an increasing function 
$\beta: (0, \infty)\to \R_+$ such that for all $t>0$ there is the on diagonal estimate $p_t(x,x)\le  {1\over \beta(t)}$
and if $\beta$ satisfies the doubling property, $\beta(2t)\le A\beta(t)$ for all $t>0$ and some number $A$,
then for some constant $D, \delta, $ and $C$,
\begin{equation}\label{off-diagonal}
p(t,x,y) \le  {C \over \beta(\delta t) }e^{-{\rho^2(x,y)\over 2Dt}}.
\end{equation}
See A. Grigoryan \cite{Grigoryan-book} and A. Bendikov and L. Saloff-Coste \cite{Bendikov-Saloff-coste-heat-kernel} for a detailed account.  In the case of $M=\R^n$, a Sobolev inequality implies Nash's inequality which in turn implies an on diagonal estimate with $\beta(t)=t^{n\over 2}$, see J. Nash \cite{Nash-inequality}.
Conversely by a theorem of N. Varopoulos \cite{Varopoulos}, generalised by E. Carlen, S. Kusuoka and D. Stroock \cite{Carlen-Kusuoka-Stroock}, the on diagonal estimate implies Sobolev's inequality.

If $\L=\sum_{k=1}^m L_{X_k}L_{X_k}+L_{X_0}$ is not elliptic, but satisfies H\"ormander condition, 
the bounds on the fundamental solution have different orders depending on 
whether the time is small or large.
To use Kolmogorov's Theorem,  it is for the small time we need the more refined upper bound. 
Under H\"ormander condition the fundamental solution $q_t$ of the parabolic equation ${\partial \over \partial t}=\L$ is expected to admit a Gaussian upper bound. For small time, it is better to use the {\it intrinsic metric distance} $d$  defined by
the formula:
$$d(x,y)=\inf\left\{l \;|\; \gamma : [0, l]\to M, \dot \gamma=\sum_{i=1}^m a_i X_i, \sum_{i=1}^m  (a_i(s) )^2\le 1\right\}.$$
where $\gamma$ is taken over all Lipschitz continuous curves on a compact interval connecting $x$ to $y$. This intrinsic distance is a
 natural distance for $\L$, i.e.  $d$ induces the original topology of the manifold. 

For diffusions on a compact manifold satisfying strong H\"ormander's conditions and with the drift $X_0$ vanishing identically, there is the following estimates in terms of the volume of the metric ball $B_x(r\sqrt t)$ centred at $x$:
\begin{equation}
\label{Gaussian-bounds}
{C_1\over \vol(B_x(\sqrt t))} e^{-{C_3 d^2(x,y)\over t}}\le q_t(x,y) \le {C_2\over \vol(B_x(\sqrt t))} e^{-{C_4 d^2(x,y)\over t}},
\end{equation}
for all $x,y\in M$ and all $t>0$. 
 This is a theorem of D. Jerison and A. Sanchez-Calle \cite{Jerison-Sanchez-Calle}. In
  A. Sanchez-Calle \cite{Sanchez-Calle78}, this upper bound is obtained for $(x,y)$ 
  satisfying the relation $d(x,y)\le \sqrt t$ and $t\le 1$. Estimates in (\ref{Gaussian-bounds}) for the heat kernel is effective
 only for small times. Indeed, as $q_t(x,y)$ is smooth and strictly positive,  we obtain trivial upper and lower constant bounds
 for $q_t$. It is another matter to obtain the best constants.
 
 For two points $x, y$ close to each other,
\begin{equation}
\label{distance-compare}
{1\over c} \rho(x,y) \le d(x,y)\le c \rho(x,y)^{1 \over l(x)},
\end{equation}
where $l(x)$ is the length in H\"ormander's condition, assuming that the intrinsic sub-Riemannian metric
associated with $\{X_1, \dots, X_m\}$ agrees with the restriction of the Riemannian metric defining $\rho$.
 If $M$ is compact and the vector fields are $BC^\infty$, then $d$ and $\rho$ are equivalent. 
 The upper bound for $d$ comes from the fact that any point in a small neighbourhood of a point $x$, of a uniform size, can be reached from $x$ by a controlled path.
This is essentially the Box-ball theorem of A. Nagel, E. Stein 
S. Wainger \cite{Nagel-Stein-Wainger}. See also R. Montgomery \cite{Montgomery-subriemannian}.
 For symmetric diffusions on $\R^n$ satisfying a `uniform H\"ormander's condition' and $t$ small,  estimates of the above form were proved in S. Kusuoka and D. Stroock \cite{Kusuoka-Stroock-III}.  For large $t$ the Euclidean metric is more relevant, see S. Kusuoka and D. Stroock \cite{Kusuoka-Stroock-Ann}.  We do not need sharp estimates on the heat kernel, however we mention that sharp estimates was obtained in  E. B. Davies \cite{Davies-sum-squares-88} 
for symmetric diffusions on a compact manifold. Also Varadhan's short time asymptotics for $\log q_t$ was given in G. Ben Arous and R. L\'eandre \cite{BenArous-Leandre-II} and R. L\'eandre \cite{Leandre-majoration,Leandre-hypoelliptic-kernel}. See also P. Friz and S. De Marco \cite{Friz-DeMarco} for a recent study.

Although an estimate of the type  (\ref{Gaussian-bounds}) is sufficient for us,  the intrinsic distance is
 not easy to use.  The fundamental solution $q_t$ is the density of the probability distribution of the $\L$-diffusion
 evaluated at $t$ with respect to the volume measure. In geodesics coordinates we easily integrate a function of $\rho$, not so easily a function of $d$.
For this reason it is convenient to use the argument that established (\ref{distance-compare}) to convert the quantities involving $d^2$ to $\rho^2$.  Let us consider the volume of the metric ball centred at $x$ with radius $\sqrt t$.
When $t$ is sufficiently small, one could apply (\ref{distance-compare}) for crude estimates. A much refined estimate
 is given in G. Ben Arous and R. L\'eandre \cite{BenArous-Leandre-II}.  For example we know that
  for $x, y$ not in each other's cut locus, as $t\to 0$
$$q_t(x,y) \sim {C(x,y)\over  t^{n\over 2}} e^{- {d^2(x,y)\over 2t}}$$
 On the diagonal $q_t(x,x) \sim c(x) t^{-{Q(x)\over 2}}$ for a number $Q(x)$ relating to $l(x)$, which holds also if $X_0$ is 
 in the span of the diffusion vector fields and their first order Lie brackets.
 G. Ben Arous and R. L\'eandre gave also an example where $X_0\not =0$ and where $q_t$ decreases exponentially on the diagonal.

\begin{lemma}
\label{hypo-bridge}
Suppose that $\L$-diffusion is conservative, has a smooth  density $q_t$ and 
\begin{enumerate}
\item For any $a_0>0$, $\sup_{ a_0\le t \le T} \sup_{x,y}q_t(x,y) <\infty$.
\item There exists  positive numbers $ \delta_0$, $a$  and $p>1$, s.t. for all $0\le s<t<T$,
\begin{equation}
{\begin{split}
&\sup_{s>{1\over 4}, |t-s| <t_0}{ \int_{M\times M}  \rho^p (x, y)q_s(x_0, x)q_{t-s}(x, y) dy dx \over  |t-s|^{1+\delta_0}} \le C;\\
&\sup_{0<t<{3\over 4}, |t-s| <t_0}
{\int_{M\times M}  { \rho^p(x,y) q_{t-s}(x, y) q_{1-t}(y, y_0) 
} dx\,dy \over  |t-s|^{1+\delta_0}} \le C.
%& \int_M\rho^p(x,y_0) q_{1-t}(x, y_0) dx \le C  |1-t|^{1+\delta_0}, \quad |1-t|<t_0
\end{split}}
\end{equation}
\end{enumerate}
Then there exist positive constants $t_0$ and $C$ such that for $|t-s|\le t_0$,  $\E \rho^p( y_s, y_t) \le C|s-t|^{1+\delta}$.
\end{lemma}
Note we do not assume the diffusion is symmetric.  By (\ref{off-diagonal}) the lemma 
applies to $\L={1\over 2} \Delta$ on a complete Riemannian manifold whose  Ricci curvature is bounded from below.  
The proof for the Lemma is included for reader's convenience. 
\begin{proof} 
We may assume $t_0<1/4$ and consider the following cases: $0<s<t<{3\over 4}$;
$0<{1\over 4}<s<t$;  $s=0$; $t=1$. We begin with the last case.
\begin{equation*}
{\begin{split}
\E \rho^p(y_s, y_0)&={1\over q (x_0,y_0)}\int_M \rho^p(x,y_0) q_s(x_0, x)q_{1-t} (x, y_0) dx\\
&\le {\sup_{s\ge {1\over 4}} \sup_{y} q_s (x, y_0) 
\over q (x_0,y_0)}\int_M \rho^p(x,y_0) q_{1-s}(x, y_0) dx.
\end{split}}
\end{equation*}
% The required conclusion follows from the third inequality in assumption (2).
   If $0<s<t<{3\over 4}$,
\begin{equation*}
{\begin{split}
\E \rho^p(y_s, y_t)&
=\int_M q_{1-t} (y, y_0)\int_M { \rho^p(x,y) q_s(x_0, x)q_{t-s}(x, y) 
 \over q (x_0,y_0)} dxdy\\
 &\le { \sup_{t<{3\over 4}} \sup_yq_{1-t} (y, y_0)\over q (x_0, y_0)}\int_M   \int_M q_s(x_0, x)  { \rho^p(x,y) q_{t-s}(x,y) 
} dydx,
\end{split}}
\end{equation*}
 concluding the estimates. The estimation for the other cases are similar.
To show that the finite dimensional distributions $q_t^{x_0, y_0}$
 determines a probability measure on $C([0,1];M)$ it is sufficient to prove that there exist $p>1$, $\delta_0>0$, and $t_0>0$
 such that if $ |t-s|<t_0$ and  $0\le s \le t\le 1$,
 $\E \rho(y_t,y_s)^p\le C|t-s|^{1+\delta_0}$. This completes the proof.
\end{proof}

 If $q$ is a continuous and $M$ is compact, assumption (1) is automatic.
 We look into condition (2) in more detail.
Denote  $\mu$ the Euclidean surface 
measure on $S^n$, $c_x(\xi)$ the distance to the cut point of $x$ along the geodesic $\gamma_x(\xi)$ in the direction of $\xi\in T_xM$. Denote $ST_xM$ the unit sphere in $T_xM$ and set
\begin{equation*}
{\begin{split}
&D_x=\{t\xi: \xi\in ST_xM, t\in[0, c(\xi))\}=T_xM\setminus C_x\\
 &D_x(r)=\{\xi\in ST_xM: r<c(\xi)\}.
\end{split}}
\end{equation*}
where $C_x$ is the Riemannian cut locus at $x$. Note that $D_x(r)$ decrease with $r$.
On $D_x$, $\exp_x$ is a diffeomorphism onto its image.  Denote $J_x(v)$ the determinant of
$(d\exp_x)_v$ identifying the tangent spaces of $T_xM$ with itself. Furthermore we denote $A_x(r)$ the lower area function:
$$A(x,r)=\int_{D_x(r)} J_x(r\xi) d\mu(\xi)={1\over r^{n-1}}\int_{D_x} J_x(\eta) d\mu(\eta).$$
%We remark that  $\vol(B_r)\le \int_0^s r^{n-1} A(x,r) dr$.

If  $A(y_0, r)$ is bounded then the last inequality in the Lemma below holds trivially. 
\begin{lemma}\label{lemma-assumption}
Suppose that there exist positive constants $C_1, C_2, C_3, \alpha, a, t_0<1$, positive increasing real valued functions $\beta_i$ decaying at most polynomially near $0$,
such that the following estimates hold for $t<t_0$,
\begin{equation*}
{\begin{split}
&q_t(x,y) \le { C_1 \over \beta_2(t) },\;
q_t(x,y) \le { C_1 \over \beta_1(t) }e^{-{ C_2\rho^{2\alpha}(x,y)\over t }} \hbox{ when }  \quad  \rho(x,y)\ge a\sqrt t;\\
& \sup_{u\ge 0} \int_{au}^\infty   r^{p+n \over \alpha}e^{-{C_2 r^2}} A(x,r^{1\over \alpha} u^{1\over \alpha})dr<\infty.
\end{split}}
\end{equation*}
Then assumption (2) of Proposition \ref{hypo-bridge} holds.
\end{lemma}

\begin{proof}
Let us consider $p>1$,  $0\le s\le t\le {3\over 4}$ and $|t-s|\le t_0$. The other cases are similar.
Working in polar coordinates we see that 
\begin{equation*}
{\begin{split}
&\int_M  q_s(x_0, x)  \int_M  { \rho^p(x,y) q_{t-s}(x,y) 
} dydx\\
=&\int_M q_s(x_0, x) \int_0^\infty   r^p \int_{D_{x}(r)}  q_{t-s}(y, \exp_{x}(r\xi))J_{x}(r\xi) \mu(d\xi) r^{n-1}dr dx.
 \end{split}}
 \end{equation*}
 We plug in the assumed upper bounds for the heat kernel in the respective regions
 to see the right hand side is bounded by:
 \begin{equation*}
{\begin{split}
&\int_M q_s(x_0, x)  \int_0^{a\sqrt {t-s}}  
 r^{n+p-1}   {C_1\over \beta_2(t-s)}\int_{D_{x}(r)}  J_{x}(r\xi) \mu(d\xi) \;dr  dx\\
 &+\int_M q_s(x_0, x){C_1\over \beta_1(t-s)}\int_{a\sqrt {t-s}} ^\infty 
 r^{n+p-1} e^{-{C_2 r^{2\alpha}\over t-s}}\int_{D_{x}(r)}   J_{x}(r\xi) \mu(d\xi)\,drdx,
 \end{split}}
 \end{equation*}
 which is further bounded by
  \begin{equation*}
 {\begin{split}
 &  {C_1\over \beta_2(t-s)}
a^{n+p-1} (t-s)^{^{n+p-1}\over 2}\int_M q_s(x_0, x) dx\int_0^{a\sqrt {t-s}}  
  A(x,r)dr  \\
 &+{C_1\over \beta_1(t-s)}\int_M  dx q_s(x_0, x)
\int_{a\sqrt{t-s}}^\infty   r^{p+n-1}e^{-{C_2 r^{2\alpha}\over t-s}} A(x,r)dr.
\end{split}}
\end{equation*}
This means,
 \begin{equation*}
 {\begin{split}
&\int_M  q_s(x_0, x)  \int_M  { \rho^p(x,y) q_{t-s}(x,y) 
} dydx\\
&\le   {C_1a^{n+p-1} (t-s)^{{n+p-1}\over 2}\over \beta_2(t-s)}
\int_M q_s(x_0, x) \int_0^{a\sqrt {t_0}}  
  A(x,r)dr  dx \\
 &+
{C_1(t-s)^{p+n\over 2\alpha}\over \beta_1(t-s)}\int_M  dx q_s(x_0, x)
\int_{a\sqrt{t-s}}^\infty   r^{p+n \over \alpha}e^{-{C_2 r^2}} A(x,r^{1\over \alpha} (t-s)^{1\over 2\alpha})dr.
\end{split}}
\end{equation*}
Since $\beta_1(t), \beta_2(t)$ decays at most polynomially neat $0$, we may choose $p$ and $\delta>0$ such that
 the assumption (2)  of Proposition \ref{hypo-bridge} holds.
\end{proof}

 The conclusions of Lemma \ref{hypo-bridge} and \ref{lemma-assumption}  hold if 
 $M=\R^n$, $\L$ satisfies the following Kusuoka-Stroock's uniform H\"ormander's condition: there exists an integer $p$ such that $l(x)\le p$. The vector fields $\{X_1, \dots, X_m\}$ and their
iterated brackets up to order $p$ give rise to a $n\times n$ symmetric matrix that is
uniformly elliptic on $\R^n$. Also $X_0$ is in the linear span of $\{X_1, \dots, X_m\}$.
In fact,   there exist constants $M>1$ and $r_0$ such that
for any $t\in (0,1]$ and $x,y\in \R^n$,  \cite{Kusuoka-Stroock-III}, the upper bound in ({Gaussian-bounds}) holds
with $C_4={1\over C_2}$.
Also the lower surface function $A(x,r)$ is bounded by a constant, the last inequality in Lemma \ref{lemma-assumption}  is satisfied.
Assumption (2) in Lemma \ref{hypo-bridge} holds.
 For $t\ge 1$, S. Kusuoka and D. Stroock  proved the following \cite{Kusuoka-Stroock-Ann},
$q_t(x,y)\le M t^{-{n\over 2}} e^{-{|y-x|^2 \over Mt}}$,
which ensures assumption (1) in Lemma \ref{hypo-bridge}.

\section{The Semi-martingale Property}
\label{section3}
 Let $x_0, z_0 \in M$ and $(y_t, 0\le t<1)$ be the solution of the following equation
\begin{equation*}
{\begin{split}
dy_t &=\sum_{i=1}^m X_i(y_t) \circ dw_t^i + X_0(y_t) dt + \nabla^H \log q_{1-t}(\cdot, z_0)(y_t)dt, \quad y_0(\omega)=x_0\\
\end{split}}
\end{equation*}
 
\begin{theorem}
\label{theorem2}
  If $M$ is compact, $X_0$ is divergence free, and $\L$ satisfies the two step strong H\"ormander condition,
  then for each $i=1, \dots, m$, 
 $$\E \int_0^1 \left|d \log q_{1-s} (\cdot, z_0)(X_i(y_s))\right| ds<\infty.$$
\end{theorem}
If $\L={1\over 2}\Delta$, this is well known.
The standard proof relies on the following estimate on the heat kernel:
$|\nabla_x \log p_t(x,y)|\le C({1\over \sqrt t}+{ \rho(x,y)\over t})$, which can be proved probabilistically or
follows from the Gaussian type upper and lower bounds
and Hamilton's estimate for the heat kernel,  R. Hamilton\cite{Hamilton93}:
$$s |\nabla_x \log p_s(\cdot, y)|^2\le C_1 \log ({C_2\over s^{n\over 2} p_s(\cdot, y_0)}).$$
 See B. Driver \cite{Driver-bb} and the following books and survey: Bismut \cite{Bismut-book}, B. Driver \cite{Driver-curved} and E. Hsu \cite{Hsu-book}.
 In B. Kim \cite[Prop. 5.2]{Kim} the following 
inequality  is proved  for a positive bounded smooth solution, satisfying further suitable $L^2$ 
estimates:
  $t|\nabla \ln u(x,t)|\le C(1+t) \ln ({M\over u(x,t)})$. There $\L$ is a `sub-elliptic' operator.  If we apply this to the kernel $q_t$, 
  together with a favourable Gaussian lower bound for $\nabla \ln u$, e.g. (\ref{Gaussian-bounds}), 
assuming that the metric balls of volume $t$ is polynomial in $t$,  we have  $$|\nabla_x \log q_t(x,y)|^2\le C ({|\ln t|\over t}+{ \rho^2(x,y)\over t^2}).$$
In terms of integrability this estimate is slightly better than the corresponding one in \cite[Lemma 5.3]{Kim}.
  However it is still on the wrong side of critical integrability at $t=0$.
  
We give some examples where the theorem holds. (1) $M=SU(2)$, and $X_1^*$, $X_2^*$ are left invariant vector fields
 generated by two Pauli matrices. (2) $M$ is the torus,  $X_1(x,y)={\partial\over \partial x}$ and
 $X_2(x,y)=\sin(2\pi x) {\partial \over \partial y}$. (3) $M=G/Z^3$ where $G$ is the Heisenberg group and 
 $X_1(x,y,z)={\partial\over \partial x}$ and $X_2(x,y,z)={\partial\over \partial y}+ x{\partial\over \partial z}$.
  
\begin{proof}
It is sufficient to prove that $\int_0^1 \sqrt{ \E |\nabla \log q_{1-s} (y_s, z_0)|^2 } ds <\infty$.
We use the following theorem of H. Cao and S. T. Yau \cite{Cao-Yau}. Let $X_0, X_1, \dots, X_m$ be smooth vector fields
on a compact manifold such that $X_0=\sum_{k=1}^m c_k X_k$  for a set of smooth real valued functions $c_k$ on $M$. Likewise suppose that for every set of $i, j, k=1, \dots, m$, $[[X_i,X_j],X_k](x)$
 can be expressed as a linear combination of vector fields from $\{X_{i'}  , [X_{j'},X_{k'}], i',j',k'=1, \dots, m\}$.
 If $u_t$ is a positive solution to the equation ${\partial \over \partial t} u_t=\sum_i L_{X_i}L_{X_i}+L_{X_0}$, there exists a
constant $\delta_0>1$, such that for all $\delta>\delta_0$  and $t>0$,
$${1\over u^2} \sum_i |L_{X_i} u|^2 \le \delta {L_{X_0} u\over u} +\delta {1\over u}{\partial u \over \partial t}+{C_1\over t} +C_2,$$
where $C_1, C_2$ are constants depending on $\L$ and $\delta_0$.
Applying this to the fundamental solution $q_t$, we see that
\begin{equation*}
\E |\nabla \log q_{1-s} (y_s, z_0)|^2
 \le
 \delta \E{L_{X_0} q_{1-s}(\cdot, z_0)\over q_{1-s}(\cdot, z_0)} (y_s)+\delta \E{{\partial q_{1-s}(\cdot, z_0) \over \partial s}(y_s)\over q_{1-s}(y_s, z_0)}
 +{C_1\over 1-s} +C_2.
\end{equation*}
Using the explicit formula for the probability density of $y_t$, we see that for any $s<1$,
\begin{equation*}
{\begin{split}
& \E\left({{\partial \over \partial s} q_{1-s}(\cdot, z_0)(y_s)\over q_{1-s}(y_s, z_0)}\right)
% =&\int_M {{\partial \over \partial s} q_{1-s}(\cdot, z_0)(x) \over q_{1-s}(x, z_0)} {q_s(x_0, x) q_{1-s} (x,z_0)\over q (x_0, z_0)}dx\\
=  \int_M {{\partial \over \partial s} q_{1-s}(x, z_0) q_s(x_0, x) \over q (x_0, z_0)}dx\\
=&\int_M {{\partial \over \partial s} (q_{1-s}(x, z_0)q_s(x_0, x))-q_{1-s}(x, z_0){\partial \over \partial s}q_s(x_0, x) \over q (x_0, z_0)}dx
=-\int_M {q_{1-s}(x, z_0){\partial \over \partial s}q_s(x_0, x) \over q (x_0, z_0)}dx.
\end{split}}
\end{equation*}
Since the divergence of $X_0$ vanishes, the same reasoning leads to the following identities:
\begin{equation*}
{\begin{split}
 &\E\left({L_{X_0} q_{1-s}(\cdot, z_0)\over q_{1-s}(\cdot, z_0)}(y_s)\right)
 %&=\int_M {L_{X_0} q_{1-s}(\cdot, z_0)(x) \over q_{1-s}(x, z_0)} {q_s(x_0, x) q_{1-s} (x,z_0)\over q (x_0, z_0)}dx\\
=  \int_M {L_{X_0} q_{1-s}(x, z_0) q_s(x_0, x) \over q (x_0, z_0)}dx\\
=&\int_M {L_{X_0} (q_{1-s}(x, z_0) q_s(x_0, x))-q_{1-s}(x, z_0) L_{X_0} q_s(x_0, x) \over q (x_0, z_0)}dx
=\int_M {-q_{1-s}(x, z_0)  L_{X_0}q_s(x_0, x) \over q (x_0, z_0)}dx
\end{split}}
\end{equation*}
Let us consider the integral from ${1\over 2}$ to $1$.
\begin{equation*}
{\begin{split}
&\int_{1\over 2}^1 \sqrt{ \E |\nabla \log q_{1-s} (y_s, z_0)|^2 } ds\\
&\le \int_{1\over 2}^1 \left( \int_M \left|{q_{1-s}(x, z_0)  (L_{X_0} q_s(x_0, x)+{\partial \over \partial s}q_s(x_0, x)) \over q (x_0, z_0)}\right|dx+{C_1\over 1-s} +C_2
  \right)^{1\over 2} \;ds
\end{split}}
\end{equation*}
Since $q_t$ is smooth and the manifold is compact, there is a constant $C_3$ such that
$$\sup_{s\in [{1\over 2},1]} \left|L_{X_0} q_s(x_0, x)+{\partial \over \partial s}q_s(x_0, x))\right|\le C_3,$$
\begin{equation*}
{\begin{split}
&\int_{1\over 2}^1 \sqrt{ \E |\nabla \log q_{1-s} (y_s, z_0)|^2 } ds
\le \int_{1\over 2}^1 \sqrt{ {C_3 \over q (x_0, z_0)} +{C_1\over 1-s} +C_2}\;ds<\infty.
\end{split}}
\end{equation*}
The same reasoning shows that $ \int_0^{1\over 2} \sqrt{ \E |\nabla \log q_{1-s} (y_s, z_0)|^2 } ds$ is finite.
\end{proof}

\begin{remark}
(1) If $\L={1\over 2}\Delta$, and $M$ is a complete Riemannian manifold with Ricci curvature is non-negative, 
there is the Harnack inequality: ${|\nabla u|^2\over u^2} - \alpha{u_t\over u}\le \alpha^2 {n\over 2t}$
where $\alpha>1$ and $C$ are constants. See P. Li and S.-T. Yau \cite{Li-Yau}  and B. Davies \cite{Davies-heat-kernel-book}.
Hence the proof of the theorem applies. 
(2) Two step Hormander condition is used in  J. Picard \cite{Picard}, for a different problem.
(3) It is also interesting to explore the Cameron-Martin quasi-invariance theorem 
in this context and prove the flow of the SDE is quasi invariant under a Girsanov-Martin shift. 
This should be straightforward if the shift is induced from vector fields of the form $\int_0^\cdot X^i(x) h_s^i ds$.
The quasi-invariance of the conditioned hypoelliptic measure is now known in some sub-Riemannian case,  see 
F. Baudoin, M. Gordina and M. Tai \cite{Baudoin-Gordina-Tai} for Heisenberg type Lie groups. 
(4) Finally we remark that a limited Li-Yau type inequality 
in  F. Baudoin and N. Garofalo \cite{Baudoin-Garofalo-14}, see also F. Baudoin, M. Bonnefont and N. Garofalo \cite{Baudoin-Bonnefont-Garofalo}, 
was extended to certain sub-Riemmanian situation, we have not yet managed to use it to our advantage, and this will be for a furture
study. A study for semigroups of H\"ormander type second order differential operators, not necessarily satisfying H\"ormander condition,
 can be found  in K. D. Elworthy, Y. LeJan and
 Xue-Mei Li \cite{Elworthy-LeJan-Li-book}.
 Finally we refer to the following articles and book for further analysis on and in sub-Riemannian geometry: 
A. Agrachev and D. Barilari \cite{Agrachev-Barilari},  
 N. Varopoulos, L. Saloff-Coste, and T. Coulhon \cite{Varopoulos-Saloff-Coste-Coulhon}, M. Bramanti \cite{Bramanti},
 A. Bellaiche \cite{Bellaiche},   D. Barilari, U. Boscain and R. Neel \cite{Barilari-Boscain-Neel} and the book by R. Montgomery \cite{Montgomery-subriemannian}. 

 \end{remark}

\def\cprime{$'$} \def\cprime{$'$}

\end{document}